\documentclass[12pt]{amsart}
\usepackage{amssymb}
\usepackage{enumerate}

        \usepackage{graphicx}
        \usepackage[usenames,dvipsnames,svgnames,table]{xcolor}
        \usepackage[a4paper]{geometry}

\newcommand{\eps}{\varepsilon}

\newcommand{\N}{{\mathbb N}}
\newcommand{\C}{{\mathbb C}}

\newcommand{\R}{{\mathbb R}}

\newcommand{\tef}{transcendental entire function}

\newcommand{\sw}{spider's web}

\newcommand\qfor{\quad\text{for }}

\def\Real{\operatorname{Re}}
\def\Imag{\operatorname{Im}}
\def\deg{\operatorname{deg}}

\theoremstyle{plain}
\newtheorem{theorem}{Theorem}[section]
\newtheorem{corollary}[theorem]{Corollary}
\newtheorem{conjecture}[theorem]{Conjecture}
\newtheorem*{theorem*}{Theorem}
\newtheorem*{proposition*}{Proposition}

\newtheorem{lemma}[theorem]{Lemma}
\theoremstyle{definition}

\theoremstyle{remark}
\newtheorem*{remark*}{Remark}
\newtheorem*{remarks*}{Remarks}

\theoremstyle{problem}

\theoremstyle{example}

\newtheorem*{example*}{Example}
\theoremstyle{question}

\theoremstyle{questions}
\newtheorem*{questions*}{Questions}

{\begin{list}{}%
         {\setlength{\leftmargin}{#1}}%
         \item[]%
}
{\end{list}}

\begin{document}


\title[Eremenko's conjecture for functions with real zeros]{Eremenko's conjecture for functions with real zeros: the role of the minimum modulus}

\author{}
\address{}
\email{}

\author{D. A. Nicks}
\address{School of Mathematical Sciences\\
The University of Nottingham\\
University Park\\
Nottingham NG7 2RD\\
UK}
\email{Dan.Nicks@nottingham.ac.uk}

\author{P. J. Rippon}
\address{School of Mathematics and Statistics \\
The Open University \\
   Walton Hall\\
   Milton Keynes MK7 6AA\\
   UK}
\email{Phil.Rippon@open.ac.uk}

\author{G. M. Stallard}
\address{School of Mathematics and Statistics \\
The Open University \\
   Walton Hall\\
   Milton Keynes MK7 6AA\\
   UK}
\email{Gwyneth.Stallard@open.ac.uk}

\thanks{The last two authors were supported by the EPSRC grant EP/H006591/1.}


\begin{abstract}
We show that for many families of {\tef}s~$f$ the property that $m^n(r)\to\infty$ as $n\to \infty$, for some $r>0$, where $m(r)=\min \{|f(z)|:|z|=r\}$, implies that the escaping set $I(f)$ of~$f$ has the structure of a spider's web. In particular, in this situation $I(f)$ is connected, so Eremenko's conjecture holds. We also give new examples of families of functions for which this iterated minimum modulus condition holds and new families for which it does not hold.
\end{abstract}
\maketitle

\section{Introduction}
\label{Intro}
\setcounter{equation}{0}

Let~$f$ be a {\tef} and denote by $f^{n},\,n=0,1,2,\ldots\,$, the $n$th iterate of~$f$. The {\it escaping set}
\[
I(f) = \{z: f^n(z) \to \infty \mbox{ as } n \to \infty\}
\]
plays a key role in complex dynamics with much recent work motivated by Eremenko's conjecture that all the components of the escaping set are unbounded. This work has led to a much better understanding of the structure of~$I(f)$.

It is known that for many families of transcendental entire functions, including the exponential family, $I(f)$ has the structure of a {\it Cantor bouquet} consisting of uncountably many unbounded curves -- see, for example, \cite{RRRS}. We have shown, however (for example in~\cite{RS10a}), that there are many families of functions for which $I(f)$ has the structure of an infinite {\it spider's web}; that is, $I(f)$ is connected and there exist bounded simply connected domains $G_n, n\in \N$, such that
\[
G_n\subset G_{n+1}\;\text{and}\; \partial G_n \subset I(f),\;\text{for } n\in\N, \quad \text{and}\quad \bigcup_{n=1}^{\infty}G_n=\C.
\]
If $I(f)$ is a Cantor bouquet or a spider's web, then Eremenko's conjecture holds.

Many of the major results on the structure of $I(f)$ have been obtained by studying the {\it fast escaping set} $A(f)$, introduced in \cite{BH99}, which can be defined as follows; see \cite{RS10a}. First put
\begin{equation}\label{ARfdef}
A_R(f) = \{z:|f^n(z)| \geq M^n(R), \text{ for } n \in \N\},
\end{equation}
where $M(r)=M(r,f)=\max \{|f(z)|:|z|=r\},\;r>0$, $M^n(r)=M^n(r,f)$ denotes the $n$th iterate of $r\mapsto M(r,f)$, and $R>0$ is so large that $M(r)>r$ for $r\ge R$, and then put
\[
A(f)=\{z: {\rm for\;some}\; \ell\in \N, f^{\ell}(z)\in A_R(f)\}.
\]

More recently we have considered the set of points whose iterates grow faster than the iterates of the {\it minimum} modulus $m(r)$, for some $r>0$, where
\[
m(r)=m(r,f)=\min \{|f(z)|:|z|=r\}.
\]
In \cite{ORS17} we showed that, if there exist $r>R>0$ such that
\begin{equation}\label{mandM}
m^n(r) > M^n(R) \mbox{ and } M^n(R) \to \infty \mbox{ as } n \to \infty,
\end{equation}
then $A_R(f)$ is a spider's web and hence $I(f)$ is a spider's web.

In~\cite{ORS17} we asked whether $I(f)$ is a spider's web whenever the much weaker condition is satisfied that there exists $r>0$ such that $m^n(r) \to \infty$ as $n \to \infty$. Here we give a positive answer to this question in the case that $f$ is a real function of finite order with only real zeros. Recall that the {\it order} $\rho(f)$ of a {\tef}~$f$ is
\[
\rho(f)=\limsup_{r\to\infty}\frac{\log\log M(r,f)}{\log r}
\]
and~$f$ is said to be {\it real} if
\[
f(\bar z)=\overline{f(z)},\qfor z\in\C.
\]
The main result of the paper gives a new condition for Eremenko's conjecture to hold.

\begin{theorem}\label{SW}
Let $f$ be a real transcendental entire function of finite order with only real zeros, and suppose
\begin{equation}\label{minmodprop}
\text{there exists } r>0 \text{ such that } m^n(r) \to \infty \text{ as } n \to \infty.
\end{equation}
Then $I(f)$ is connected and, moreover, is a spider's web.
\end{theorem}

{\it Remarks.} 1. All functions of order less than 1/2 have the property~\eqref{minmodprop}, along with several other families of functions (see~\cite[Theorem~1.1]{ORS17}), but many of these functions do not satisfy \eqref{mandM}. In particular, there are functions of order less than~1/2 that satisfy the hypotheses of Theorem~\ref{SW} but do not satisfy \eqref{mandM}; indeed, such examples can even be of order~0 (see~\cite{RS13b}).

2. Many previously known examples for which $I(f)$ is a spider's web have the stronger property that $A_R(f)$ is a spider's web (see \cite{RS10a}). But there are also examples for which this is not the case. For example, in~\cite{RS13c}, we showed that $I(f)$ is a spider's web for real functions of order less than 1/2 whose zeros are all on the negative real axis; the examples of functions of order~0 mentioned above lie in this class and for these functions $A_R(f)$ is {\it not} a spider's web (see~\cite{RS13b}). More recently, Evdoridou~\cite{V} showed that $I(f)$ is a spider's web but $A_R(f)$ is not a spider's web for Fatou's function $f(z) = z + 1 + e^{-z}$.

The proof of Theorem~\ref{SW} is in two parts, depending on the genus of the function, a concept to be defined shortly which is closely related to the order of the function. In Section~2 we prove Theorem~\ref{SW} for all functions satisfying the hypotheses of the theorem with genus at most 1. In fact, we show that, for such functions, $I(f)$ contains a spider's web of points that escape to infinity at a uniform rate related to the minimum modulus. In particular, as we show in Section~3, for real functions of order less than 1/2 with real zeros, a certain subset of $I(f)$ called the quite fast escaping set, $Q(f)$, contains a spider's web.

The first part of the proof of Theorem~\ref{SW} uses a new approach that was introduced in \cite[Theorem 1.1]{RS13c} to prove that $I(f)$ is a spider's web for any real function of order less than~$1/2$ with only negative zeros. The structure of the proof here is similar to that in \cite{RS13c}, though the techniques used within the proof are substantially different. Here the iterated minimum modulus plays a key role, together with new results on the winding of image curves that we obtained in~\cite[Theorem 1.1]{NRS} in order to prove Baker's conjecture for real functions with real zeros.

In Section~4, we use results from the value distribution theory of entire functions to show that, if~$f$ is a transcendental entire function of finite order with only real zeros and of genus at least~2, then there is {\it no} value of~$r$ for which $m^n(r) \to \infty$ as $n\to \infty$, and so the hypotheses of Theorem~\ref{SW} cannot be satisfied. In fact we prove the stronger result that, for such functions, there exists a ray from~0 on which $f(z) \to 0$ as $z \to \infty$. This result, which is of independent interest, completes the proof of Theorem~\ref{SW}.

The genus of a function is defined as follows. Recall that any finite order entire function has a Hadamard representation
\begin{equation}\label{Hadamard}
f(z) =  z^n e^{Q(z)} \prod_{k=1}^\infty E(z/a_k, m),
\end{equation}
where~$Q$ is a polynomial, the $a_k$ are the non-zero zeros of $f$, the Weierstrass primary factors are
\[
E(z,0)=1-z\quad\text{and}\quad E(z,m)=(1-z)e^{z+\frac{z^2}{2}+\cdots+\frac{z^m}{m}},\;\;m\ge 1,
\]
and $m$ is the smallest integer for which $\sum_{k=1}^\infty |a_k|^{-(m+1)}$ is convergent. The \emph{genus} of~$f$ is
\begin{equation}\label{genus}
\max\{ m, \deg Q\}.
\end{equation}
The genus is thus an integer and it satisfies the inequalities
\[
\rho(f) -1 \le \text{genus of }f \le \rho(f);
\]
see \cite[pp. 24--29]{MF} or \cite[pp. 250--253]{T}, for example.

In Section~\ref{1/2to2}, we study real entire functions with only real zeros, and of genus less than~2 and order at least $1/2$. We construct examples of such functions for which property \eqref{minmodprop} holds and examples for which it does not hold.

A summary of our results relating the property \eqref{minmodprop} to the growth of real {\tef}s with only real zeros is given in the following theorem.

\begin{theorem}\label{3cases}
\begin{itemize}
\item[(a)]
Let~$f$ be a real {\tef} of finite order with only real zeros.
\begin{itemize}
\item[(i)]
If $0\le \rho(f)<1/2$, then property \eqref{minmodprop} always holds.
\item[(ii)]
If the genus of~$f$ is at least~2 (which includes the case that $\rho(f)>2$), then property \eqref{minmodprop} never holds.
\end{itemize}
\item[(b)]
For any $\rho\in [1/2,2]$ there are examples of real {\tef}s with real zeros that have order~$\rho$ for which property \eqref{minmodprop} holds and also examples of such functions for which \eqref{minmodprop} does not hold.
\end{itemize}
\end{theorem}

In view of the results in Theorem~\ref{3cases}, it is natural to ask whether property~\eqref{minmodprop} holds for some or all entire functions of order~1/2 minimal type. We shall investigate this question in forthcoming work~\cite{NRS18}.

\section{Spiders' webs in $I(f)$}
\label{genus1}
\setcounter{equation}{0}


We begin our proof of Theorem~\ref{SW} by letting $f$ be a real transcendental entire function of genus at most~1 with only real zeros for which there exists $r>0$ such that $m^n(r) \to \infty$ as $n \to \infty$. We show that, for such a function, $I(f)$ is a spider's web. In fact we prove the following stronger result.

\begin{theorem}\label{genus1theorem}
Let~$f$ be a real transcendental entire function of genus at most~1 with only real zeros and suppose that there exists $r>0$ such that
\begin{equation}\label{strongminmodprop}
m^n(r) \text{ is a strictly increasing sequence with } m^n(r) \to \infty \text{ as } n \to \infty,
\end{equation}
and $m(r)^{1/3} \ge |f(0)|$. Then $I(f)$ contains a spider's web in
\[
 \{ z: |f^n(z)| \geq  M^{-1}(m^{n}(r)^{1/3}) \mbox{ for all } n \in \N \}.
\]
\end{theorem}

This result is sufficient to show that Theorem~\ref{SW} holds for functions of genus at most~1. Indeed, whenever $I(f)$ contains a spider's web, $I(f)$ must be a spider's web, by~\cite[Lemma 4.5]{RS13c}, and whenever \eqref{minmodprop} holds there exists $r>0$ such that \eqref{strongminmodprop} holds, by \cite[Theorem~2.1]{ORS17}, which is stated as Lemma~\ref{fullequiv} in this paper.

The proof of Theorem~\ref{genus1theorem} uses the sequence of closed sets
\[
I_N = \{ z: |f^n(z)| \geq  M^{-1}(m^{n+N}(r)^{1/3}) \mbox{ for all } n \in \N \},\quad N=0,1,2,\ldots,
\]

and we show that $I_N$ contains a spider's web, for all $N=0,1, \ldots,$ using proof by contradiction. The basic idea of the proof is that if $I_N$ does {\it not} contain a spider's web for some~$N$, then we can take a suitably long curve $\gamma_0$ in its complement. We then show that successive images of this curve must either experience repeated radial stretching escaping to infinity, or they must eventually wind round the origin meeting a component of the fast escaping set. In either case we are able to deduce that the curve $\gamma_0$ contains a point in $I_N$ which is a contradiction.

The proof requires several results from earlier papers. Before we state these we introduce some notation.

For $r>0$, we write $C(r) = \{z: |z| = r\}$ and, for $0<r_1<r_2$, we write
\[
A(r_1,r_2)=\{z:r_1<|z|<r_2\}\;\;\text{and}\;\;\overline A(r_1,r_2)=\{z:r_1\le |z|\le r_2\}.
\]

If~$\gamma$ is a plane curve that lies in a simply connected domain containing no zeros of~$f$, then, for any pair of distinct points $z_0,z'_0$ in~$\gamma$, we denote the net change in the argument of $f(z)$ as $z$ traverses~$\gamma$ from $z_0$ to $z'_0$ by $\Delta\!\arg (f(\gamma);z_0,z'_0)$.

The main tool that we use to obtain winding is a corollary of the following theorem, proved in~\cite[Theorem 2.1]{NRS}. Note that~\cite[Theorem 2.1]{NRS} was stated for a continuum but here we only need to apply the result to a curve.

Also, although~\cite[Theorem 2.1]{NRS} was stated for certain entire functions of order less than 2, the proof only required~$f$ to be in the Laguerre--P\'olya class (see \cite[discussion in Section~5]{NRS}). This class is the closure of the set of real polynomials with real zeros and, by the Laguerre--P\'olya theorem (\cite{eL} and \cite{gP}), it consists of functions of the form
\[
f(z) =  \pm z^n e^{b_2z^2+b_1z+b_0} \prod_{k=1}^\infty E(z/a_k, m),
\]
where the $a_k$ are real, $m=0$ or $1$, $b_0, b_1\in\R$ and $b_2\le 0$. In particular, all real entire functions with real zeros and genus at most~1 belong to this class.

\begin{theorem}\label{main1}
Let $f$ be a real transcendental entire function of genus at most~1 with only real zeros. There exists $R_0=R_0(f) > 0$ such that, if~$s$ and~$a$ are positive real numbers with
\begin{equation}\label{s>}
s\ge R_0 \quad\text{and}\quad \log s \ge \frac{64}{a^2} + \frac{80\pi}{a}\
\end{equation}
and $\gamma$ is a curve in $\{z:\Imag z \ge 0\}$ that meets both $C(s)$ and $C(s^{1+a})$ with
\begin{equation}
1/M(s) \le |f(z)| \le M(s), \quad\mbox{for } z\in\gamma, \label{|f|<M}
\end{equation}
then there exist a curve $\Gamma\subset\gamma\cap \overline{A}(s,s^{1+a})$ and $z_0,z'_0\in\Gamma$ such that
\[ \Delta\!\arg(f(\Gamma);z_0,z'_0) \ge \frac{1}{10\pi}\log M(s) \log s^a. \]
\end{theorem}

The corollary of Theorem~\ref{main1} which we use is given below. This result essentially appears in the middle of the proof of~\cite[Theorem 2.4]{NRS} but we include the proof here for simplicity.

\begin{corollary}\label{winding}
Let $f$ be a real transcendental entire function of genus at most~1 with only real zeros. There exists $R_1=R_1(f)\ge R_0(f)$ such that, if~$t$ and~$\eps$ are positive real numbers with
\begin{equation}\label{teps}
t^{1-\eps} \ge R_1 \quad\text{and}\quad \eps \in [10 / \sqrt{\log t}, 1),
\end{equation}
and $\gamma$ is a curve in $\{z:\Imag z \ge 0\}$ that meets both $C(t^{1-\eps})$ and $C(t)$ with
\begin{equation}\label{Meps}
1/M(t^{1-\eps}) \le |f(z)| \le M(t^{1-\eps}), \quad\mbox{for } z\in\gamma,
\end{equation}
then there exist a curve
$\Gamma\subset\gamma\cap \overline{A}(t^{1-\eps},t)
$
and $z_0,z'_0\in\Gamma$ such that
\[ \Delta\!\arg(f(\Gamma);z_0,z'_0) \ge 2\pi. \]
\end{corollary}
\begin{proof}
We claim that we can apply Theorem~\ref{main1} to $\gamma$ with $s=t^{1-\eps}$ and $s^{1+a} = t$. It follows from~\eqref{Meps} that~\eqref{|f|<M} holds. To show that~\eqref{s>} holds, we note that $a=\eps/(1-\eps)$ and so, by \eqref{teps},
\begin{eqnarray*}
\frac{64}{a^2} + \frac{80 \pi}{a}  & = & \frac{64 (1-\eps)^2}{\eps^2} + \frac{80 \pi (1-\eps)}{\eps}\\
& < & \left( \frac{64}{100}\log t + \frac{80 \pi}{10} \sqrt{\log t} \right) (1-\eps)\\
& = & \left( \frac{64}{100} + \frac{8 \pi}{\sqrt{\log t}} \right) (1-\eps) \log t\\
& < & (1-\eps)\log t = \log s,
\end{eqnarray*}
provided $\sqrt{\log t} > 200 \pi / 9$. Thus ~\eqref{s>} holds, provided $R_1$ is sufficiently large. So it follows from Theorem~\ref{main1} that
there exist a curve
\[
\Gamma\subset\gamma\cap \overline{A}(s,s^{1+a}) = \gamma\cap \overline{A}(t^{1-\eps},t)
\]
 and $z_0,z'_0\in\Gamma$ such that
\begin{eqnarray*}
\Delta\!\arg(f(\Gamma);z_0,z'_0) & \ge & \frac{1}{10\pi}\log M(t^{1-\eps}) \log t^{\eps} \\
 & \ge & \frac{1}{10\pi}\log M(t^{1-\eps}) \log (10 \sqrt{\log t})\\
 & \ge & \frac{1}{10\pi}  \log (10 \sqrt{\log t}) \geq 2\pi,
\end{eqnarray*}
by \eqref{teps}, provided $R_1$ is sufficiently large.
\end{proof}

We also use the following result about the fast escaping set, proved in \cite[Lemma~4.4]{RS13c}.

\begin{lemma}\label{AR-component}
Let $f$ be a {\tef}. There exists $R_2=R_2(f)>0$ such that if $R\ge R_2$, then there is a component of $A_{R/2}(f)$ that meets $\{z:|z|<R\}$ and is unbounded.
\end{lemma}

The next result we need concerns uniform rates of escape of quite a general form. This result was proved in~\cite[Theorem~1.4]{V}. Note that, although the statement of~\cite[Theorem~1.4]{V} assumes that the sequence $(a_n)$ satisfies $a_{n+1} \leq M(a_n)$, for $n \in \N$, the proof there only uses the consequence of this assumption that $a_{n} \leq M^n(R)$ for $n \in \N$ and some $R>0$, and we now state the result in that form.

\begin{lemma}\label{Ifan}
Let $f$ be a {\tef} and let $(a_n)$ be a positive increasing sequence with $a_n \to \infty$ as $n \to \infty$, $a_{n} \leq M^n(R)$ for $n \in \N$ and some $R>0$, and $a_1$ sufficiently large that the disc $D(0,a_1)$ contains a periodic cycle of $f$. Let
\[
I(f,(a_n)) = \{z: |f^n(z)| \geq a_n \mbox{ for all } n \in \N\}.
\]
If $I(f,(a_n))^c$ has a bounded component, then $I(f,(a_n))$ contains a spider's web.
\end{lemma}

We also use the following result on the convexity of the maximum modulus function; see \cite[Lemma~2.2]{RS09}.

\begin{lemma}\label{convex}
Let $f$ be a {\tef}. There exists $R_3=R_3(f)>0$ such that if $r>R_3$ and $c>1$ then
\[
M(r^c) \geq M(r)^c.
\]
\end{lemma}

Finally, we use the following technical lemma.
\begin{lemma}\label{prodL}
Let $r_n$ be a sequence satisfying
\begin{equation}\label{r}
r_0 \geq \exp(1600) \quad\mbox{and}\quad r_{n+1} \geq r_n,\quad \mbox{for } n \geq 0.
\end{equation}

Suppose further that there exists a subsequence $r_{n_k}$ such that
\begin{equation}\label{rnk}
r_{n_{k+1}} \geq r_{n_k}^{16},\quad \mbox{for } k \in \N.
\end{equation}

Now let $(L_n)$ be a sequence such that
\begin{equation}\label{L}
L_0 = 3, \;\;  L_{n_k+1} = L_{n_k}(1-\delta_{n_k}), \mbox{ for } k \in \N, \;\;\mbox{and}\;\; L_n = 3 \mbox{ otherwise},
\end{equation}
where
\begin{equation}\label{delta}
\delta_{n_k} = 10/\sqrt{\log r_{n_k}},\quad \mbox{for } k \in \N.
\end{equation}
Then $L_n \geq 2$ for all $n \in \N$.
\end{lemma}

\begin{proof}
It follows from~\eqref{L} that
\begin{equation}\label{prod}
L_n \geq 3 \prod_{k \in \N} (1 - \delta_{n_k}),\quad \mbox{for } n \in \N.
\end{equation}
Also, it follows from~\eqref{r} and~\eqref{rnk} that
\[
\log r_{n_k} \geq 16^{k-1} \log r_{n_1} \geq 16^{k-1} \log r_0 \geq 100 \times 16^k,\quad \mbox{for } k \in \N.
\]
Together with~\eqref{delta}, this implies that
\begin{equation}\label{deltank}
\delta_{n_k} \leq \frac{10}{\sqrt{100 \times 16^k}} = \frac{1}{4^k},\quad \mbox{for } k \in \N.
\end{equation}
It follows from~\eqref{prod} and~\eqref{deltank} that
\[
L_n \geq 3 \prod_{k \in \N}\left(1 - \frac{1}{4^k}\right)\geq 3 \left(1-\sum_{k \in \N} \frac{1}{4^k}\right) = 2,\quad \mbox{for } n \in \N,
\]
as claimed.
\end{proof}

We now prove the main result of this section.

\begin{proof}[Proof of Theorem~\ref{genus1theorem}]
First recall that $r>0$ satisfies~\eqref{strongminmodprop} and $m(r)^{1/3}\ge |f(0)|$.

As stated earlier, we shall show that, under the given hypotheses, all the sets
\begin{equation}\label{INdef}
I_N = \{ z: |f^n(z)| \geq  M^{-1}(m^{n+N}(r)^{1/3}) \mbox{ for all } n \in \N \},\quad N=0,1,2,\ldots,
\end{equation}
contain a spider's web, which is sufficient to prove Theorem~\ref{genus1theorem}.

We assume then that there exists some positive integer~$N_0$ such that the set $I_{N_0}$ does {\it not} contain a spider's web, and show that this assumption gives a contradiction. Since, by \eqref{strongminmodprop}, the sets $I_N$ have the property that
\begin{equation}\label{INnested}
I_{N_2}\subset I_{N_1},\quad\text{for }N_2>N_1\ge 0,
\end{equation}
our assumption implies that, for all $N\ge N_0$, the set $I_{N}$ does not contain a spider's web.

We now choose~$N\ge N_0$ so large that
\begin{equation}\label{Nprop1}
\text{the disc }D(0,M^{-1}(m^{N+1}(r)^{1/3})) \text{ contains a periodic cycle of } f,
\end{equation}
\begin{equation}\label{Nprop2}
m^{N+2}(r)^{1/3} \ge \max\{R_1, R_3, M(R_2),M(1), \exp(1600)\}
\end{equation}
and
\begin{equation}\label{Nprop3}
M(s)\ge s^{32},\qfor s\ge \left(m^{N+2}(r)\right)^{1/2}.
\end{equation}
Next we choose~$R$ so large that if $a_n = M^{-1}(m^{n+N}(r)^{1/3})$, then $a_n\le M^n(R)$ for $n\in\N$. Indeed, if we choose $R \ge m^{N-1}(r)$, then for all $n\in\N$ we have
\[
m^{n+N}(r)=m^{n+1}\left(m^{N-1}(r)\right)\le M^{n+1}(R),\quad\text{so}\quad a_n\le M^n(R).
\]
Thus such a choice of~$R$ is possible. Also, by \eqref{Nprop1}, the disc $D(0,a_1)$ contains a periodic cycle of~$f$.

Then the hypotheses of Lemma~\ref{Ifan} are satisfied for this sequence $(a_n)$. Therefore, since $I_N$ does not contain a spider's web, $I_N^c$ must have at least one unbounded component and hence there must be an unbounded curve $\Gamma_0 \subset I_N^c$.

Next note that if we increase $N$, then \eqref{Nprop1}--\eqref{Nprop3} remain true, as does the statement that $\Gamma_0\subset I_N^c$ by \eqref{INnested}.

We now increase~$N$ if necessary to ensure that the curve $\Gamma_0$ meets both the circles $C(M^{-1}(m^{N+2}(r)^{1/3}))$ and  $C(m^{N+2}(r))$. We can then choose a subcurve $\gamma_0$ of $\Gamma_0$ such that
\[
\gamma_0\subset \overline A(M^{-1}(m^{N+2}(r)^{1/3}),m^{N+2}(r))
\]
and
\[
\gamma_0 \text{ meets both } C(M^{-1}(m^{N+2}(r)^{1/3})) \text{ and } C(m^{N+2}(r)).
\]
The idea of the proof is to show that if such a curve $\gamma_0$ exists, then we can construct a sequence of curves $\gamma_n\subset f^n(\gamma_0)$ and positive sequences $(r_n)$ and $(L_n)$ such that, for $n \geq 0$, we have
\begin{equation}\label{contains}
f(\gamma_{n})\supset \gamma_{n+1},
\end{equation}
\begin{equation}\label{annulus}
\gamma_n\subset \overline A(M^{-1}(r_n^{1/L_n}),r_n),\text{ and }\gamma_n\text{ meets both }C(M^{-1}(r_n^{1/L_n})) \text{ and } C(r_n),
\end{equation}
where
\begin{equation}\label{rL}
r_n \geq m^{N+n+2}(r)\quad \mbox{and}\quad 2 \leq L_n \leq 3,
\end{equation}
and also
\begin{equation}\label{rL2}
L_{n+1}=3 \quad \mbox{or}\quad L_{n+1} \geq L_n(1 - 10/\sqrt{\log r_n})\mbox{ and } r_{n+1} \geq M(r_n^{1/2}) \geq r_n^{16}.
\end{equation}
(These conditions formalise what we mean by saying that the images of the curve~$\gamma_0$ experience `repeated radial stretching'.)

We can then deduce from~\eqref{contains},~\eqref{annulus},~\eqref{rL}, and~\eqref{rL2}, that there is a point $z_0\in \gamma_0$ such that, for $n\ge 0$,
\begin{equation}\label{gamman}
f^n(z_0)\in \gamma_n,\quad \text{so}\quad |f^n(z_0)|\ge M^{-1}( r_n^{1/3}).
\end{equation}
By \eqref{INdef} and \eqref{rL}, this implies that $z_0 \in I_N$ which is a contradiction since $z_0 \in \gamma_0 \subset I_N^c$.

To construct the sequences~$(L_n)$,~$(r_n)$ and~$(\gamma_n)$, we proceed as follows. Start by putting $r_0=m^{N+2}(r)$, so $\gamma_0$, $r_0$ and $L_0=3$ have the required properties.
Next, suppose that, for $k=0,1,\ldots, n$, we have chosen curves $\gamma_k$ and positive numbers $r_k$ and $L_k$ such that~\eqref{annulus} and~\eqref{rL} are satisfied with $n$ replaced by $k$ and, in addition,~\eqref{contains} and~\eqref{rL2} hold, with $n$ replaced by $k$ for $k=0,1,\ldots,n-1$. We then show that we can choose a curve $\gamma_{n+1}$ and positive numbers $r_{n+1}$ and $L_{n+1}$ so that~\eqref{contains} and~\eqref{rL2} hold, and~\eqref{annulus} and~\eqref{rL} hold with $n$ replaced by $n+1$.

We begin by putting
\begin{equation}\label{rnmax}
r_{n+1} = \max_{z \in \gamma_n}|f(z)|
\end{equation}
and let~$q$ be the largest integer such that $m^q(r) \leq r_n$. It follows from~\eqref{rL} that $q \geq N+n+2$. Also, it follows from~\eqref{annulus} that $\gamma_n$ must meet $C(m^q(r))$, since
\[
m^q(r)\ge M^{-1}(m^{q+1}(r)) > M^{-1}(r_n)\ge M^{-1}(r_n^{1/L_n}).
\]
Thus
\begin{equation}\label{rn1}
r_{n+1} \geq m^{q+1}(r) \geq m^{N+n+3}(r),
\end{equation}
so $r_{n+1}$ satisfies the first inequality in~\eqref{rL}.

Next we suppose that there exists $z \in \gamma_n$ with $|f(z)| \leq M^{-1}(r_{n+1}^{1/3})$. Then it is clear that there exists a curve $\gamma_{n+1}$ having the properties \eqref{contains}--\eqref{rL2}, with $L_{n+1} = 3$.

If such a~$z$ does not exist, then we must have
\begin{equation}\label{frn1}
|f(z)| > M^{-1}(r_{n+1}^{1/3}),\quad \mbox{for } z \in \gamma_n.
\end{equation}

We now put
\begin{equation}\label{eps}
\eps_n = 10/ \sqrt{\log r_n}
\end{equation}
and note that, by \eqref{Nprop2} and \eqref{rL}, we have
\begin{equation}\label{eps12}
\eps_n \le 1/2.
\end{equation}
We consider two cases. First suppose that
\begin{equation}\label{epsn}
r_{n+1} \geq M(r_n^{1 - \eps_n}).
\end{equation}
In this case we have $r_{n+1} > M(r_n^{1/2}) \geq r_n^{16}$, by \eqref{eps12}, \eqref{rL} and \eqref{Nprop3}. We put
\begin{equation}\label{Ln1}
L_{n+1} = L_n(1 - \eps_n)
\end{equation}
and note that $r_{n+1}$ and $L_{n+1}$ satisfy \eqref{rL2}. Also, by \eqref{Nprop2}, it follows from Lemma~\ref{prodL} that $L_{n+1} \geq 2$. Thus $L_{n+1}$ satisfies the second inequality in~\eqref{rL}.

We can now show that there exists a curve $\gamma_{n+1}$ with the required properties. First, it follows from~\eqref{annulus} that there exists $z \in \gamma_n$ with
\begin{equation}\label{gn1}
|f(z)| \leq r_n^{1/L_n}.
\end{equation}
 Also, it follows from~\eqref{epsn}, Lemma~\ref{convex} (in view of \eqref{Nprop2} and \eqref{rL}), and~\eqref{Ln1} that
\begin{equation}\label{gn2}
r_{n+1} \geq M(r_n^{1-\eps_n}) \geq M(r_n^{1/L_n})^{L_n(1-\eps_n)} = M(r_n^{1/L_n})^{L_{n+1}},
\end{equation}
so $r_n^{1/L_n}\le M^{-1}(r_{n+1}^{1/L_{n+1}})$. Together with~\eqref{gn1} and \eqref{rnmax}, this is sufficient to show that there exists a curve $\gamma_{n+1}$ satisfying~\eqref{contains} and~\eqref{annulus} with~$n$ replaced by~$n+1$.

It remains to consider the case when
\begin{equation}\label{large}
r_{n+1} < M(r_n^{1 - \eps_n}).
\end{equation}
We show that this leads to a contradiction and so cannot occur. A key fact needed to obtain this contradiction is that $f$ has the symmetry property
\begin{equation}\label{sym}
f(\overline{z}) = \overline{f(z)}, \quad\mbox{for } z \in \C.
\end{equation}

We will obtain a contradiction by applying Corollary~\ref{winding}, with
\[
t=r_n \quad\mbox{and}\quad \eps=\eps_n,
\]
to a curve $\gamma'_n$ meeting $C(r_n^{1- \eps_n})$ and $C(r_n)$, chosen such that $\gamma'_n \subset \{z: \Imag z \geq 0\}$ and
\begin{equation}\label{gamma}
\gamma'_n \subset \gamma_n \cup \gamma_n^*,
 \end{equation}
where $*$ denotes reflection in the real axis.

We check that the hypotheses of Corollary~\ref{winding} are satisfied. First we have
\[
r_n^{1-\eps_n}\ge r_n^{1/2} \ge R_1,
\]
by \eqref{Nprop2} and \eqref{rL}. Next, it follows from~\eqref{rnmax}, \eqref{large}, \eqref{sym} and~\eqref{gamma} that
\[
|f(z)| \leq M(r_n^{1-\eps_n}),\quad \mbox{for } z \in \gamma_n',
\]
and, from~\eqref{frn1}, \eqref{sym}, \eqref{gamma}, and also \eqref{rn1} and \eqref{Nprop2}, that
\begin{equation}\label{bound}
|f(z)| \geq M^{-1}(r_{n+1}^{1/3}) > 1 > 1/M(r_n^{1-\eps_n}),\quad \mbox{for } z \in \gamma_n'.
\end{equation}
Therefore, by Corollary~\ref{winding}, there exists a curve $\Gamma \subset \gamma'_n$ and points $z_0, z_0' \in \Gamma$ and such that
\begin{equation}\label{twopi}
 \Delta\!\arg(f(\Gamma);z_0,z'_0) \ge 2\pi.
 \end{equation}
We also have, from~\eqref{bound}, that
\[
|f(z)| \geq M^{-1}(r_{n+1}^{1/3}), \quad\mbox{for } z \in \Gamma.
\]
Together with~\eqref{twopi},~\eqref{Nprop2}, \eqref{sym} and Lemma~\ref{AR-component}, this implies that
\[
f(\Gamma) \cap A_{M^{-1}(r_{n+1}^{1/3})/2}(f) \neq \emptyset.
\]
So, by~\eqref{sym} and \eqref{gamma}, there exists $z_n \in \gamma_n$ such that
\begin{equation}\label{fznfast}
|f(z_n)| \geq M^{-1}(r_{n+1}^{1/3})\quad\mbox{and}\quad f(z_n) \in A_{M^{-1}(r_{n+1}^{1/3})/2}(f).
\end{equation}
From the construction of the curves $\gamma_k$, for $1 \leq k \leq n$, it follows that there exists $z_0 \in \gamma_0$ such that $f^k(z_0) \in \gamma_k$, for $0 \leq k \leq n$, and $f^n(z_0) = z_n$. Therefore, by \eqref{annulus}, \eqref{rL} and \eqref{rn1}, we have
\[
|f^k(z_0)| \geq M^{-1}(r_{k}^{1/3}) \geq M^{-1}(m^{N+k+2}(r)^{1/3}),\quad \mbox{for } 0 \leq k \leq n+1,
\]
and, by \eqref{fznfast}, \eqref{ARfdef}, \eqref{Nprop3} and then \eqref{rn1}, we have
\[
|f^{n+2}(z_0)| = |f^2(z_n)| \geq M(M^{-1}(r_{n+1}^{1/3})/2) \geq M^{-1}(r_{n+1}^{1/3}) \geq M^{-1}(m^{N+n+3}(r)^{1/3}).
\]
By similar reasoning, we have, for $k > n+2$,
\begin{align*}
|f^{k}(z_0)| &  = |f^{k-n-1}(f(z_n))| \geq M^{k-n-1}(M^{-1}(r_{n+1}^{1/3})/2)\ge M^{n+k-2}(M^{-1}(r_{n+1}^{1/3}))\\
&\geq M^{k-n-2}(M^{-1}(m^{N+n+3}(r)^{1/3})) = M^{-1}(M^{k-n-2}(m^{N+n+3}(r)^{1/3})),
\end{align*}
and hence, by \eqref{Nprop3},
\[
|f^{k}(z_0)|\geq M^{-1}(M^{k-n-3}(m^{N+n+3}(r))) \geq M^{-1}(m^{N+k}(r)) > M^{-1}(m^{N+k}(r)^{1/3}).
\]
Together these three estimates show that $z_0 \in I_N$, which is a contradiction since $z_0 \in \gamma_0 \subset I_N^c$. It follows that the case considered in~\eqref{large} cannot occur and this completes the proof.
\end{proof}

\section{Spiders' webs in $V(f)$ and $Q(f)$}
\label{VandQ}
\setcounter{equation}{0}

In this section we prove two results which follow from Theorem~\ref{genus1theorem} and which show that, for many classes of real transcendental entire functions with only real zeros, the sets $V(f)$ and $Q(f)$ each contain a spider's web. These subsets of the escaping set were considered in earlier papers in connection with Eremenko's conjecture and a conjecture of Baker. In particular, we show here that for all real entire functions of order less than 1/2 with only real zeros, the quite fast escaping set $Q(f)$ contains a spider's web.

We introduced the set $V(f)$ in \cite{ORS17}. This is the set of points whose iterates grow at least as fast as the iterated minimum modulus; that is,
\[
V(f) = \{ z: \mbox{ there exists } L \in \N \mbox{ such that } |f^{n+L}(z)| \geq \tilde{m}^n(R) \mbox{ for all } n \in \N \},
\]
where
\[
\tilde{m}(r) = \max_{0\leq s \leq r} m(s),\quad r>0,
\]
and $R>0$ is any number such that $\tilde{m}(r) > r$ for $r \geq R$. The existence of such an~$R$ follows from \eqref{minmodprop} by a result about escaping points of real functions \cite[Theorem~2.1]{ORS17}, which we have already used at the start of Section~\ref{genus1} and which we state here in full for the reader's convenience.
\begin{lemma}
\label{fullequiv}
Let $ \varphi: [0, \infty) \to [0, \infty) $ be continuous and put $\tilde{\varphi}(t) = \max_{0\leq s \leq t} \varphi(s)$, $t>0$ . Then the following statements are equivalent.
\begin{enumerate}[(a)]
\item There exists $ t > 0 $ such that $ \varphi^n(t) \to \infty $ as $ n \to \infty. $
\item There exists $ t' > 0 $ such that the set $ \lbrace \varphi^n(t'): n \in \N_0 \rbrace $ is unbounded.
\item There exists $ T>0 $ such that $ \widetilde{\varphi}(t) > t $, for $ t \geq T $ (or equivalently there exists~$t>0$ such that $\widetilde{\varphi}^n(t) \to \infty$ as $n\to\infty$).
\item There exist $ t \geq T > 0 $ such that
\[ \varphi^n(t) \textrm{ and } \widetilde{\varphi}^{\, n}(T) \textrm{ increase strictly with } n \textrm{ to } \infty, \]
and
\[ \varphi^n(t) \in [ \, \widetilde{\varphi}^{\, n}(T) , \widetilde{\varphi}^{\, n+1}(T) \, ], \qfor n \in \N_0. \]
\item There exists a sequence $ (t_n) $ of positive real numbers such that $ t_n \to \infty $ as $ n \to \infty $ and
\begin{equation*}
\varphi(t_n) \geq t_{n+1}, \quad \textrm{ for }  n \in \N_0.
\end{equation*}
\end{enumerate}
\end{lemma}

By the equivalence of parts~(a) and~(c) of Lemma~\ref{fullequiv}, property \eqref{minmodprop} holds if and only if there exists~$R>0$ such that $\tilde{m}(r) > r$ for $r \geq R$. Moreover, in this situation there must exist $r>0$ such that $m^n(r) > \tilde{m}^n(R)$ for $n \in \N$ by part~(d).

As mentioned in the introduction, we also showed in \cite{ORS17} that there are large classes of functions for which there exist $r>R>0$ such that
\[
m^n(r) > M^n(R) \mbox{ and } M^n(R) \to \infty \mbox{ as } n \to \infty.
\]
For such functions, $V(f)$ is equal to the fast escaping set $A(f)$ and this set is a spider's web. It is known, however, that this equality is not true in general for functions such that \eqref{minmodprop} holds, even for functions of order less than 1/2. We asked in~\cite{ORS17} whether $V(f)$ contains a spider's web for all functions such that \eqref{minmodprop} holds.

As a consequence of Theorem~\ref{genus1theorem}, we have the following partial result in this direction.

\begin{theorem}\label{V(f)}
Let $f$ be a real transcendental entire function of genus at most~1 with only real zeros. If there exist $R>0$ and $p \in \N$ such that
\begin{equation}\label{Mandtildem}
\tilde{m}^p(s) \geq M(s),\quad \text{for } s \geq R,
\end{equation}
then $V(f)$ contains a spider's web.
\end{theorem}
\begin{proof}
First, it follows from \eqref{Mandtildem} that we can assume that $\tilde m^n(R)\to\infty$ as $n\to\infty$. Hence, by the equivalence of parts~(c) and~(d) of Lemma~\ref{fullequiv}, there exists $r \geq R$ such that $m^n(r)$ is strictly increasing and $m^n(r) \geq \tilde{m}^n(R)$ for $n \in \N$.

If we now take $N \geq 2p+2$, sufficiently large that $M(r)\ge r^3$ for $r\ge M(\tilde m^{N-2p}(R))$, then, for $n \in \N$, we have
\begin{align*}
M^{-1}(m^{n+N}(r)^{1/3}) &\geq M^{-1}(\tilde{m}^{n+N}(R)^{1/3})\\
&\geq M^{-1}((M^2(\tilde{m}^{n+N-2p}(R))^{1/3})\\
&\geq M^{-1}(M(\tilde{m}^{n+N-2p}(R))\\
&= \tilde{m}^{n+N-2p}(R).
\end{align*}
So it follows from Theorem~\ref{genus1theorem} that $V(f)$ contains a spider's web.
\end{proof}

We end this section by showing that, in the special case that~$f$ is a real transcendental entire function with only real zeros and order less than 1/2, Theorem~\ref{V(f)} together with the cos\,$\pi\rho$ theorem imply that the {\it quite fast escaping set} $Q(f)$ contains a spider's web. Recall that, for each $\eps > 0$, we define
\[
Q_{\eps}(f) = \{ z: \mbox{ there exists } L \in \N \mbox{ such that } |f^{n+L}(z)| \geq \mu_{\eps}^n(R) \mbox{ for all } n \in \N \},
\]
where
\[
\mu_{\eps}(r) = M(r)^{\eps},\quad r>0,
\]
and $R>0$ is so large that $\mu_{\eps}(r)>r$, for $r\ge R$, and then put
\[
Q(f) = \bigcup_{\eps \in (0,1)} Q_{\eps}(f).
\]

In an earlier paper \cite{RS13c}, we saw other families of functions of order less than 1/2 for which $Q(f)$ contains a spider's web. For many of these, $Q(f) = A(f)$ (see \cite{RS13a}), but there are examples for which $Q(f)$ contains $A(f)$ strictly (see \cite{RS13b}).

We know of no examples of functions of order less than 1/2 for which $Q(f)$ does not contain a spider's web, and we make the following conjecture.

\begin{conjecture}
If~$f$ is a transcendental entire function of order less than 1/2, then the quite fast escaping set $Q(f)$ contains a spider's web, from which it follows that $I(f)$ is a spider's web and hence is connected.
\end{conjecture}

We have the following partial result in this direction.

\begin{theorem}\label{Q(f)}
Let $f$ be a transcendental entire function of order less than~$1/2$. Then $V(f) \subset Q(f)$. If, in addition,~$f$ is
real with only real zeros, then $Q(f)$ contains a spider's web.
\end{theorem}
\begin{proof}
Let~$f$ be a transcendental entire function of order less than 1/2. It follows from the $\cos\pi \rho$ theorem (see \cite{iB58}, \cite{Ba63} or \cite[page~331]{wH89}) that there exists $\eps \in (0,1)$ such that, for sufficiently large $r > 0$,
\[
\text{there exists }s\in (r^{\eps},r)\text{ such that } m(s)\ge M(r^{\eps}).
\]
Therefore, for sufficiently large $r>0$, we have
\begin{equation}\label{mtildeM}
\tilde{m}(r) \geq  M(r^{\eps}) \text{ and hence } \tilde{m}^n(r) \ge (\mu_{\eps}^n(r^{\eps}))^{1/\eps}\ge \mu_{\eps}^n(r^{\eps}),\qfor n\in\N.
\end{equation}
Hence $V(f) \subset Q(f)$. The fact that $Q(f)$ contains a spider's web now follows from Theorem~\ref{V(f)}, by using \eqref{mtildeM} with $n=2$ and Lemma~\ref{convex}.
\end{proof}

\section{Functions of genus at least~2}
\label{ge2}
\setcounter{equation}{0}
In this section we complete the proof of Theorem~\ref{SW} by showing that if~$f$ is a real {\tef} of finite order with only real zeros and genus at least~2, then its minimum modulus
\[
m(r)\to 0\;\text{ as } r\to \infty,
\]
so~$f$ does not satisfy \eqref{minmodprop}. Actually, we prove the following much stronger result.
\begin{theorem}\label{main2}
Let~$f$ be a {\tef} of finite order with only real zeros and genus at least~$2$. Then
\begin{itemize}
\item[(a)] there exists $\theta\in [0,2\pi]$ such that
\[
f(re^{i\theta})\to 0 \;\text{ as } r\to\infty;
\]
\item[(b)] 0 is a deficient value of~$f$.
\end{itemize}
Both {\rm (a)} and {\rm (b)} imply that $m(r)\to 0$ as $r\to\infty$, so $f$ does not satisfy \eqref{minmodprop}.
\end{theorem}
The conclusion that~0 is a deficient value states that the {\it defect} of~$f$ at 0,
\begin{equation}\label{defect}
\delta(0,f)= 1-\limsup_{r\to\infty}\frac{N(r)}{T(r)}>0.
\end{equation}
Here,
\[
N(r)=N(r,0,f)=\int_0^r \frac{(n(t,f)-n(0,f))}{t}\,dt,
\]
in which $n(t,f)$ is the number of zeros of~$f$ in $\{z:|z|\le t\}$ counted according to multiplicity, and
\[
T(r) = N(r)+ \frac{1}{2\pi}\int_0^{2\pi} \log^+(1/|f(re^{i\theta})|)\,d\theta,
\]

where $\log^+t=\max\{\log t,0\}$. It is clear that if 0 is a deficient value of~$f$, then $m(r)\to 0$ as $r\to\infty$, since
\[
\liminf_{r\to\infty}\left(\frac{1}{2\pi T(r)}\int_0^{2\pi} \log^+(1/|f(re^{i\theta})|)\,d\theta\right)=\liminf_{r\to\infty}\left(1-\frac{N(r)}{T(r)}\right)>0,
\]
so
\[
\frac{1}{2\pi}\int_0^{2\pi} \log^+(1/|f(re^{i\theta})|)\,d\theta\to\infty\quad\text{as }r\to\infty.
\]

Theorem~\ref{main2} has the following corollary for functions with only non-negative zeros.
\begin{corollary}\label{cor-main2}
Let~$f$ be a {\tef} of finite order with only non-negative zeros and genus at least~$1$. Then
\begin{itemize}
\item[(a)] there exists $\theta\in [0,2\pi]$ such that
\[
f(re^{i\theta})\to 0 \;\text{ as } r\to\infty;
\]
\item[(b)] 0 is a deficient value of~$f$.
\end{itemize}
Hence $m(r)\to 0$ as $r\to\infty$, so $f$ does not satisfy \eqref{minmodprop}.
\end{corollary}
\begin{proof}
The function $g(z)=f(z^2)$ has finite order, its zeros all lie on the real axis, and its genus is at least~2. Hence, by Theorem~\ref{main2},~$g$ has limit~0 along a ray to~$\infty$, and~0 is a deficient value of~$g$. It follows immediately that~$f$ has limit~0 along a ray to~$\infty$, and also that~0 is a deficient value of~$f$, since, for $r\ge t>0$, we have
\[
n(t,g)=2n(t^2,f),\quad\text{so}\quad N(r,0,g)=N(r^2,0,f),
\]
and
\[
\frac{1}{2\pi}\int_0^{2\pi} \log^+(1/|g(re^{i\theta})|)\,d\theta=2\left(\frac{1}{2\pi}\int_0^{2\pi} \log^+(1/|f(r^2e^{i\theta})|)\,d\theta\right).\qedhere
\]
\end{proof}
The proof of Theorem~\ref{main2} uses the following two lemmas. The first, used in the proof of part~(b), is a major result of Edrei, Fuchs and Hellerstein \cite[Corollary~1.2]{EFH}.

\begin{lemma}\label{EFH1.2}
Suppose that~$f$ is an entire function with only real zeros, $a_n$ say, such that
\begin{equation}\label{zeros-order}
\sum_{n=1}^{\infty}\frac{1}{|a_n|^2} = \infty\quad\text{and}\quad \sum_{n=1}^{\infty}\frac{1}{|a_n|^\xi} < \infty,
\end{equation}
for some $\xi\in (2,\infty)$. Then~0 is a deficient value of~$f$.
\end{lemma}
The second lemma, used in the proof of part~(a), concerns the asymptotic behaviour of Weierstrass primary factors of the form
\[
E(z,m)=(1-z)e^{z+\frac{z^2}{2}+\cdots+\frac{z^m}{m}},\quad m\ge 2.
\]
\begin{lemma}\label{lemma:E}
Given an integer $m\ge 2$, there exists an angle $\theta$ such that
\begin{equation} \label{theta-choice}
\begin{split}
&\bullet \mbox{ if~$m$ is even, then $\cos m\theta <0$ and $\cos(m+1)\theta = 0$;}  \\
&\bullet \mbox{ if~$m$ is odd, then $\cos (m-1)\theta <0$, $\cos m\theta=0$ and $\cos(m+1)\theta > 0$.}
\end{split}
\end{equation}
Moreover, for such $\theta$ and all $T\in \R$, we have
\begin{equation}\label{E-bound}
 |E(Te^{i\theta},m)| \le 1,
\end{equation}
and there exist $C, T_0>0$ such that, if $T\in\R$ and $|T|>T_0$, then
\begin{equation}\label{logE-bound}
\log|E(Te^{i\theta},m)| \le \begin{cases} -C|T|^m, &\mbox{if $m$ is even,} \\
-C|T|^{m-1}, &\mbox{if $m$ is odd.} \end{cases}
\end{equation}
\end{lemma}

\begin{proof}
First, for $m\ge 2$ and~$m$ even, and any angle of the form
\[
\theta=\frac{(4k-1)\pi}{2(m+1)},\quad k=1,2,\ldots, \tfrac12 m,
\]
we have
\[
(m+1)\theta=2k\pi-\frac{\pi}{2} \quad\text{and}\quad
m\theta= \left(2k\pi-\frac{\pi}{2}\right) \frac{m}{m+1}\in \left(2k\pi-\frac{3\pi}{2},2k\pi-\frac{\pi}{2}\right).
\]
The even case of \eqref{theta-choice} follows from this, with~$m$ choices of the angle $\theta\in [0,2\pi]$, by adding~$\pi$ to each of the $\tfrac12 m$ choices above.

The odd case of \eqref{theta-choice} follows similarly with
\[
\theta=\frac{(4k-1)\pi}{2m},\quad k=1,2,\ldots, \tfrac12(m-1),
\]
giving $m-1$ choices of $\theta$, by adding~$\pi$ to each of these $\tfrac12 (m-1)$ choices. It is helpful in the odd case to note that when $\cos m\theta=0$, the conditions $\cos(m-1)\theta<0$ and $\cos(m+1)\theta>0$ are equivalent.

For any integer $m\ge 2$, such an angle $\theta$, and $T\in \R$, we have
\begin{equation}\label{log|E|}
  \log |E(Te^{i\theta},m)| = \Real\left(Te^{i\theta} + \frac{(Te^{i\theta})^2}{2} + \cdots + \frac{(Te^{i\theta})^m}{m}\right) + \log|1-Te^{i\theta}|.
\end{equation}
Now,
\begin{eqnarray*}
  \frac{d}{dT} \log |E(Te^{i\theta},m)| &=& \Real\left(e^{i\theta}\left(\frac{(Te^{i\theta})^m -1}{Te^{i\theta}-1}\right)\right) + \frac{T-\cos\theta}{(1-T\cos\theta)^2+T^2\sin^2\theta} \\
   &=& \frac{T^{m+1}\cos m\theta - T^m\cos(m+1)\theta}{(1-T\cos\theta)^2+T^2\sin^2\theta}\,.
\end{eqnarray*}
By our choice of $\theta$, this last expression is positive when $T<0$ and is negative when $T>0$. It follows that $\log |E(Te^{i\theta},m)| \le \log |E(0,m)| = 0$ for all $T\in \R$.

We also see from \eqref{log|E|} that, as $|T|\to\infty$,
\[ \log |E(Te^{i\theta},m)| = \frac{T^m\cos m\theta}{m} + \frac{T^{m-1}\cos (m-1)\theta}{m-1} + O(|T|^{m-2}). \]
The final claim of the lemma follows, again using our choice of $\theta$.
\end{proof}

\begin{proof}[Proof of Theorem~\ref{main2}]
We first write~$f$ in the form \eqref{Hadamard}, where the zeros $a_k$ are real and~$Q$ is a polynomial of degree $d_Q$. The proof when~$f$ has a finite number of zeros is a simpler version of the proof when $f$ has infinitely many zeros, so we assume the latter.

We then write
\[
P(z)=\prod_{k=1}^\infty E(z/a_k, m),\quad\text{so}\quad f(z)=z^ne^{Q(z)}P(z),
\]
and we recall that~$m$ is the least integer for which $\sum |a_k|^{-(m+1)}$ is convergent. We also recall from \eqref{genus} that the genus of~$f$ is $\max\{m, d_Q\}$.

First we prove part~(a). The proof splits into two cases depending on whether or not $d_Q$ is greater than~$m$. In some sense, when $m\ge d_Q$ the growth/decay of $P(z)$ dominates that of $e^{Q(z)}$, while the reverse is true when $m<d_Q$.

Suppose first that $m\ge d_Q$. Then $m\ge 2$ by the hypothesis on the genus of~$f$.


Let~$\theta$ be as in Lemma~\ref{lemma:E}. Then there exists $\theta_0\in\{\theta, \theta+\pi\}$ and $A>0$ such that for $z_0=re^{i\theta_0}$ with $r>1$, we have
\[
\log |z_0^ne^{Q(z_0)}| \le \begin{cases} Ar^m, &\mbox{if $m$ is even,} \\
Ar^{m-1}, &\mbox{if $m$ is odd.} \end{cases}
\]

To see that this holds when $m$ is odd, note that $\Real(Q(re^{i\theta_0}))$ is a polynomial in~$r$ of degree at most~$m$. If this degree equals~$m$, then $\theta_0\in\{\theta,\theta+\pi\}$ can be chosen to make the leading coefficient of $\Real(Q(re^{i\theta_0}))$ negative, so in this case $\log |z_0^ne^{Q(z_0)}|<0$ for large~$r$.

Next consider $\log |P(z_0)|$. Let~$C$ and $T_0$ be as in Lemma~\ref{lemma:E} and write $\alpha=1/T_0$. By \eqref{E-bound} and \eqref{logE-bound},
\begin{eqnarray*}
  \log|P(z_0)| &=& \sum_{k=1}^\infty \log\left|E\left(\frac{re^{i\theta_0}}{a_k},m\right)\right|  \\
   &\le& \sum_{|a_k|<\alpha r} \log\left|E\left(\frac{re^{i\theta_0}}{a_k},m\right)\right| \le
   \begin{cases} \ \ -Cr^m {\displaystyle\sum_{|a_k|<\alpha r} |a_k|^{-m}}, &\mbox{if $m$ is even,} \vspace{2mm}\\
-Cr^{m-1} {\displaystyle\sum_{|a_k|<\alpha r} |a_k|^{-(m-1)}}, &\mbox{if $m$ is odd.} \end{cases}
\end{eqnarray*}
Since these two sums diverge as $r\to\infty$, it follows from the estimates above that
\[ \log |f(re^{i\theta_0})| = \log |z_0^ne^{Q(z_0)}| + \log |P(z_0)| \to -\infty \; \mbox{ as } r\to\infty,\]
as required.

Now suppose instead that $d_Q\ge m+1$. We first observe that
\begin{equation}\label{log|P|}
  \log|P(z)| = o(|z|^{m+1}) \; \mbox{ as } z\to\infty.
\end{equation}


The proof of \eqref{log|P|} depends only on the fact that $\sum |a_k|^{-(m+1)}<\infty$, and not on the fact that these zeros are real. Indeed, this convergence implies (see \cite[p.17 and Lemma~1.4]{MF}) that $n(r,P)$ is of at most order~$m+1$ convergence class; that is,
\[
\int_0^{\infty}\frac{n(t,P)}{t^{m+2}}\,dt<\infty,\;\text{ so }\; n(r,P)=o(r^{m+1})\;\text{ as }r\to\infty.
\]
The estimate~\eqref{log|P|} then follows from \cite[Theorem~1.11]{MF} (see in particular the antepenultimate sentence of its proof) or alternatively  from \cite[p.~233,~(3.3)]{N}.

Since~$Q$ is a polynomial, there exist $\theta\in [0,2\pi]$ and $c=c(\theta)>0$ such that
\[
|e^{Q(re^{i\theta})}| < \exp(-cr^{d_Q}),\quad \text{for sufficiently large } r>0.
\]
In fact an estimate of this type holds for all $\theta$ in a union of open subintervals of $[0,2\pi]$ of total length~$\pi$. We conclude, using \eqref{log|P|} and the assumption that $d_Q\ge m+1$, that for such~$\theta$ we have
\[ |f(re^{i\theta})| \le r^n \exp((-c+o(1))r^{m+1}) \to 0\;\text{ as } r\to \infty, \]
as required.

The proof of part~(b) also involves two cases.

First suppose that $m\ge 2$. Then we can simply apply Lemma~\ref{EFH1.2} to deduce that~0 is a deficient value of~$f$.

Next suppose that $m\le 1$. Then $d_Q\ge 2$ by our hypothesis about the genus, so  $d_Q\ge m+1$. Now we argue as in the proof of \eqref{log|P|} to deduce that
\[
n(r,f)=o(r^{m+1})\;\text{ as }r\to\infty,\quad\text{so}\quad N(r)=o(r^2)\;\text{ as }r\to\infty.
\]
Since $d_Q\ge 2$, the function~$f$ has order at least~2 and, moreover,
\[
\liminf_{r\to\infty} \frac{T(r)}{r^2}>0.
\]
We deduce that $N(r)=o(T(r))$ as $r\to\infty$, and hence
\begin{equation}\label{defect1}
\delta(0,f)= 1-\limsup_{r\to\infty}\frac{N(r)}{T(r)}=1.
\end{equation}
Therefore~0 is again a deficient value, in this case with defect~1.

This completes the proof of Theorem~\ref{main2}.
\end{proof}

We conclude this section by mentioning another family of {\tef}s for which we can show that $m(r) \to 0$ as $r \to \infty$,  and hence that property \eqref{minmodprop} does not hold.

\begin{theorem}\label{inf-order}
Let~$f$ be a {\tef} of infinite order, with zeros $a_n$, $n\in \N$, such that
\begin{equation}\label{zeros-order}
\sum_{n=1}^{\infty}\frac{1}{|a_n|^\xi} < \infty,
\end{equation}
for some $\xi\in (0,\infty)$. Then $\delta(0,f)=1$, so $m(r)\to 0$ as $r\to\infty$, and hence~$f$ does not satisfy~\eqref{minmodprop}.
\end{theorem}
\begin{proof}
It follows from the argument used to prove \eqref{log|P|} that $N(r)$ has finite order. On the other hand~$f$ has infinite order so we can write
\[
f(z)=z^ne^{h(z)}P(z),
\]
where $n\ge 0$, $h$ is a {\tef}, and~$P$ is the Hadamard product associated with~$f$, which has finite order. It follows that~$f$ has infinite lower order; see \cite[Proof of Theorem~3]{tK76}, for example. Hence \eqref{defect1} holds, so $\delta(0,f)=1$ and hence $m(r) \to 0$ as $r \to \infty$ by the reasoning following the statement of Theorem~\ref{main2}.
\end{proof}

In view of Theorems~\ref{main2} and~\ref{inf-order}, it is natural to ask whether it is the case that \eqref{minmodprop} fails to hold if~$f$ is {\it any} {\tef} of infinite order with only real zeros. For a general {\tef}~$f$ of infinite order with its zeros lying on a finite number of rays from~0, Miles \cite[Theorem~1]{jM79} has proved that there is a set $E\subset [0,\infty)$ of zero logarithmic density such that
\[
\lim_{r\to\infty, r\notin E}\frac{N(r)}{T(r)}=0,\quad\text{so } \lim_{r\to\infty, r\notin E} m(r) = 0.
\]
On its own, however, this property is not sufficient to show that \eqref{minmodprop} fails to hold.
%

\section{Functions with order in the interval $[1/2,2]$}
\label{1/2to2}
\setcounter{equation}{0}
We have seen that if~$f$ is any {\tef} of order less than~$1/2$, then \eqref{minmodprop} holds, and if~$f$ is a finite order function with real zeros of genus at least~2 (which includes the case that $\rho(f)>2$), then \eqref{minmodprop} does not hold. In this section, we start by giving examples of real {\tef}s with all their zeros on the positive real axis, all of genus~0 and having all possible orders in the interval $[1/2,1]$. For each order in this interval, we shall give one example that does satisfy \eqref{minmodprop} and another example that does not.

We can then adapt these examples to construct examples of functions of genus~1 with only real zeros. Indeed, for real {\tef}s~$f$ with only positive zeros, of genus~0 and order $\rho$ say, the function $g(z)=f(z^2)$ is real with all real zeros, of genus~1 and having order $2\rho$, and \eqref{minmodprop} holds for~$g$ if and only if it holds for~$f$. Therefore, for each possible order in $[1,2]$, we obtain one example of this type that does satisfy \eqref{minmodprop} and another example that does not.


First, for order $1/2$ the following examples were mentioned in \cite[Example~8.4]{ORS17}:
\begin{itemize}
\item property (\ref{minmodprop}) does not hold for the function $ f(z) = \cos \sqrt z $,
\item property (\ref{minmodprop}) holds for the functions $ g(z) = 2z \cos \sqrt z $.
\end{itemize}
Both these functions are real and have their zeros on the positive real axis, and are of order $1/2$ and genus~$0$.

Next, we give functions of genus~0 and all orders in $(1/2,1)$ for which \eqref{minmodprop} does {\it not} hold. Consider the family of functions
\begin{equation}\label{Hardy}
f_{\sigma}(z)=\prod_{n=1}^{\infty}\left(1-\frac{z}{n^{\sigma}}\right),
\end{equation}
where $\sigma=1/\rho$ and $\frac12<\rho<1$. Then $f_{\sigma}$ is a real {\tef} of genus 0 and order $\rho$, with zeros at $n^{\sigma}$, $n\in\N$. Hardy \cite{gH05} showed that
\[
f_{\sigma}(z)\sim \frac{2}{\sqrt{2\pi z}}\sin(\pi z^{\rho})\exp(\pi\cot(\pi\rho)z^{\rho}),
\]
as~$z$ tends to $\infty$ within a domain that contains the positive real axis. In particular, since $\cot(\pi\rho)<0$ for $\frac12<\rho<1$, we have $m(r,f_{\sigma})\to 0$ as $r\to \infty$, so \eqref{minmodprop} fails for these functions.

For an example of a real {\tef}~$f$ of order~1 and genus~$0$ with only positive zeros for which \eqref{minmodprop} fails, we note that Lindel\"of \cite{eL02} showed that the functions
\[
f_{\alpha}(z)=\prod_{n=1}^{\infty}\left(1-\frac{z}{n(\log n)^{\alpha}}\right), \quad\text{where } 1<\alpha<2,
\]
satisfy
\[
f_{\alpha}(z)=\exp\left(\frac{1+o(1)}{1-\alpha}z(\log(-z))^{1-\alpha}\right),
\]
as $z$ tends to $\infty$ within any set of the form $\{z:|\arg z|<\pi-\delta\}$, $\delta>0$. It is easy to deduce from this estimate that each such $f_{\alpha}$, $1<\alpha<2$, is bounded on all rays of the form $\{z:\arg z=\theta\}$, where $\theta\in (0,\pi/2)$. Hence $m(r,f_{\alpha})$ is bounded, so \eqref{minmodprop} fails for these functions.

The zeros of all the functions considered so far in this section are distributed very evenly on the positive real axis. With a more uneven distribution of positive zeros, we can construct real {\tef}s given by infinite products for which \eqref{minmodprop} {\it does} hold.

In the proof of the following result we show that \eqref{minmodprop} holds for the constructed function~$f$ by again using the fact that \eqref{minmodprop} is equivalent to the existence of $R=R(f)>0$ such that
\begin{equation}\label{tildem}
\tilde m(r) = \max_{0\leq s \leq r} m(s) >r, \quad\text{for }r\ge R;
\end{equation}
see Lemma~\ref{fullequiv}.

\begin{theorem}\label{dan1}
Given any $\rho$, $0<\rho\le 1$, we can construct a function of the form
\begin{equation}\label{prodf}
f(z)=\prod_{k=1}^{\infty}\left(1-\frac{z}{a_k}\right)^{m_k},
\end{equation}
where $a_k>0$ and $m_k\in\N$ for $k\in\N$, such that~$f$ has order $\rho$ and genus~$0$, and \eqref{minmodprop} holds.
\end{theorem}

\begin{proof} We first assume that $0<\rho<1$. We shall construct the sequence of zeros $(a_k)$ to be strictly increasing, with multiplicity $(m_k)$, and have several other properties. Let $n(r)$ be the number of such zeros in $\{z:|z|\le r\}$, counted according to multiplicity. One condition we require is that the sequence $(m_k)$ is chosen in such a way that
\[
n(a_k)=\sum_{j=1}^{k}m_j=[a_k^{\rho}],\quad k\in\N,
\]
where $[.]$ denotes the integer part function. This ensures that the infinite product in \eqref{prodf} is convergent and~$f$ has order $\rho$ by \cite[Theorem~1.11 and Lemma~1.4]{MF}.

We now describe how to choose $(a_k)$. First we estimate $m(r_k) = |f(r_k)|$, where $r_k=3a_k$. We have
\begin{equation}\label{min-rk1}
\left|\prod_{j=1}^{k}\left(1-\frac{r_k}{a_j}\right)^{m_j}\right|\ge \prod_{j=1}^{k}2^{m_j}=2^{[a_k^{\rho}]},
\end{equation}
 and we shall choose the sequence~$(a_k)$ so large that
\begin{equation}\label{min-rk2}
\left|\prod_{j=k+1}^{\infty}\left(1-\frac{r_k}{a_j}\right)^{m_j}\right|\ge \frac12.
\end{equation}
This can be achieved by choosing~$(a_k)$ such that
\begin{equation}\label{min-rk3}
\sum_{j=k+1}^{\infty}\frac{m_jr_k}{a_j}\le\frac12.
\end{equation}
Since $m_k \le a_k^{\rho}$ for $k\in\N$, we have
\[
\sum_{j=k+1}^{\infty}\frac{m_jr_k}{a_j}\le r_k\sum_{j=k+1}^{\infty}a_j^{\rho-1}.
\]
If we take
\begin{equation}\label{choose-ak}
a_0=1, \quad a_{k+1}=(12 a_k^{\rho})^{1/(1-\rho)},\quad k=0,1,\ldots,
\end{equation}
then
\[
a_j\ge \left(12a_k^{\rho}\right)^{(j-k)/(1-\rho)}a_k=T_k^{(j-k)/(1-\rho)}a_k,\qfor j\ge k+1,
\]
say, where $T_k=12a_k^{\rho}$. Hence
\[
r_k\sum_{j=k+1}^{\infty}a_j^{\rho-1}\le r_k a_k^{\rho-1}\sum_{j=k+1}^{\infty}\left(\frac{1}{T_k}\right)^{j-k}=\frac{3a_k^{\rho}}{T_k-1}\le \frac12,
\]
which proves \eqref{min-rk3}, and also shows that~$f$ has genus~$0$. On combining \eqref{min-rk1} and \eqref{min-rk2}, we obtain
\[
m(r_k)=|f(r_k)|\ge 2^{[a_k^{\rho}]-1},\qfor k\ge 1.
\]
Finally, we prove that \eqref{tildem} holds. Given~$r$ we choose~$k$ such that $r_k\le r<r_{k+1}$. Then
\[
\tilde m(r)\ge m(r_k) \ge 2^{[a_k^{\rho}]-1}>3 (12a_k^{\rho})^{1/(1-\rho)} =3a_{k+1}=r_{k+1}>r,
\]
provided that $k$ is sufficiently large, as required.

The proof when $\rho=1$ is a modification of the argument above in which~$f$ remains of the form \eqref{prodf} but we take
\[
n(a_k)=\sum_{j=1}^k m_j=[a_k^{1-\eps_k}],\quad\text{where } \eps_k=\frac1k,\qfor k\in\N,
\]
which implies that $f$ has order~1, and we replace \eqref{choose-ak} by
\[
a_0=1, \quad a_{k+1}=(12 a_k)^{1/\eps_{k+1}},\quad k=0,1,\ldots.
\]
Then $a_k\ge 12^k$ for $k\in \N$, so
\[
a_j^{\eps_j} \ge 12^{j-k}a_k \ge 4^{j-k} r_k,\qfor j\ge k+1,
\]
where $r_k=3a_k$ as before. It readily follows, by splitting the product as above, that
\[
m(r_k)=|f(r_k)|\ge 2^{[a_k^{1-\eps_k}]-1},\qfor k\ge 1.
\]
Hence, for $r_k\le r<r_{k+1}$ and $k$ sufficiently large,
\[
\tilde m(r)\ge m(r_k) \ge 2^{a_k^{1/2}}>3 (12a_k)^{1/\eps_{k+1}} =3a_{k+1}=r_{k+1}>r,
\]
since $a_k\ge 12^k$ for $k\in\N$.
\end{proof}

\end{document}